\documentclass[12pt]{article}

\usepackage{amsmath}
\usepackage{latexsym}
\usepackage{amssymb}
\usepackage{ifthen}
\usepackage{tikz}
\usepackage{amsthm}
\usepackage{yfonts}

\newcommand{\junk}[1]{}

\newcommand{\formal}[1]{\ensuremath{\textsf{#1}}}

\newcommand{\dword}[1]{\textbf{#1}}

\newcommand{\clnb}[1]{\ensuremath{\formal{N}[#1]}}
\newcommand{\clnbk}[1]{\ensuremath{\formal{N}_{#1}}}

\newcommand{\graph}[1]{\ensuremath{\mathbf{#1}}}

\newcommand{\capt}{\ensuremath{\formal{capt}}}

\newcommand{\multfunct}{\ensuremath{\longrightarrow}}
\newcommand{\rfunct}[1]{\ensuremath{\mathbf{cr}(#1)}}

\newcommand{\ctime}{\ensuremath{\mathfrak{c}  }}

\newcommand{\uprank}[2]{\ensuremath{#1^{[#2 ]}}}
\newcommand{\X}{\ensuremath{\formal{V}}}

\newtheorem{defn}{Definition}[section]
\newtheorem{thm}[defn]{Theorem}

\newtheorem{cor}[defn]{Corollary}

\newtheorem{lemma}[defn]{Lemma}

\newtheorem*{convention}{Convention}

\title {Cop-Win Graphs:  Optimal Strategies and Corner Rank}

\author{David Offner\\
\small Department of Mathematics and Computer Science\\[-0.8ex]
\small Westminster College\\[-0.8ex] 
\small New Wilmington, PA, U.S.A.\\
\small\tt offnerde@westminster.edu\\
\and
Kerry Ojakian\\
\small Department of Mathematics and Computer Science\\[-0.8ex]
\small Bronx Community College (CUNY)\\[-0.8ex]
\small Bronx, NY, U.S.A\\
\small\tt kerry.ojakian@bcc.cuny.edu
}

\begin{document}
\maketitle




\begin{abstract}

We investigate the game of \emph{cops and robber}, played on a finite graph, between one cop and one robber.   If the cop can force a win on a graph,  the graph is called \emph{cop-win}.
We describe a procedure we call \emph{corner ranking}, performed on a graph, which assigns a positive integer or $\infty$ to each vertex.  
We give a characterization of cop-win in terms of corner rank 
and also show that the well-known characterization of cop-win via dismantling orderings follows from our work.
From the corner rank we can determine the capture time of a graph, i.e.
the number of turns the cop needs to win. 
We describe a class of optimal cop strategies we call \emph{Lower Way} strategies, and a class of optimal robber strategies we call \emph{Higher Way} strategies.
Roughly speaking, in a Lower Way strategy, the cop pushes the robber down to lower ranked vertices, while in a Higher Way strategy, the robber moves to a highest rank vertex that is ``safe.''
While interesting in their own right, the strategies are themselves tools in our proofs.  
We investigate various properties of the Lower Way strategies.

\end{abstract}

\section{Introduction}

The game of cops and robber is a perfect-information two-player pursuit-evasion game played on a  graph.  To begin the game, the cop and robber each choose a vertex to occupy, with the cop choosing first.  Play then alternates between the cop and the robber, with the cop moving first. On a turn a player may move to an adjacent vertex or stay still.  If the cop and robber ever occupy the same vertex, the robber is caught and the cop wins.  If the cop can force a win on a graph, we say the graph is \emph{cop-win}. The game was introduced by Nowakowski and Winkler \cite{NW83}, and Quilliot \cite{Qui78}. A nice introduction to the game and its many variants is found in 
the book by Bonato and Nowakowski \cite{BN11}. 

For a cop-win graph, the cop can guarantee a win, so
Bonato et. al.
 \cite{BGHK2009}
raise the question of how long the game will last.
They define the \dword{capture time} of a cop-win graph \graph{G}, denoted $\capt(\graph{G})$, to mean the fewest number of moves \emph{the cop} needs to guarantee a win (\emph{not counting her initial placement as a move}); for example, on the path with five vertices the capture time is two. 
For cop-win graphs, we define  what we mean for a player's strategy to be optimal.

\begin{defn} Consider a cop-win graph with capture time
\ctime.
\begin{itemize}

\item A cop strategy is \dword{optimal} if no matter what the 
robber does, the cop will win in at most \ctime \ cop moves.

\item A robber strategy is \dword{optimal} if no matter what the  cop does, the 
cop must make at least \ctime \ moves to win.

\end{itemize}

\end{defn}

Our paper is centered on the development of optimal strategies which are defined in terms of \emph{corner rank}.
In Sections \ref{sec_corner_ranking} and \ref{sec_projections},
we introduce corner rank and the associated notion
of projections.  On any graph, we define a corner rank function which
assigns to each vertex a positive integer or $\infty$. The corner rank (or simply rank) of a vertex is the value assigned to it by the corner rank function.
The corner rank of a graph is equal to the rank of its highest rank vertex.
In Section~\ref{LowWay} we define a class of cop strategies we call \emph{Lower Way} strategies;
roughly, these strategies corner the robber among the vertices of at least
a certain rank, forcing the robber to move to successively 
lower rank vertices until he is caught.
In Section~\ref{OneStep} we define a class of robber strategies we call \emph{Higher Way};
in such a strategy, the robber moves to a highest corner rank vertex
he can while maintaining a condition guaranteeing his safety.

We say a graph with a corner rank that is a positive integer has
\emph{finite corner rank}, while a graph with a corner rank of $\infty$ has \emph{infinite corner rank}.
If the corner rank of a graph is finite then we can easily determine a number
\ctime \ such that  a Lower Way strategy will win in at most \ctime \ cop moves.
Conversely, the Higher Way strategy forces the cop to make at least \ctime \ moves.
Thus we can conclude that the capture time is \ctime \ and that both the 
Lower Way strategies and Higher Way strategies are optimal.
While a Higher Way strategy for the robber on a finite corner rank graph delays his inevitable loss
as long as possible, on an infinite corner rank graph, the strategy actually wins.
Thus corner rank provides a characterization of cop-win graphs: A graph is cop-win if and only
if it has finite corner rank.

Finally, in Section~\ref{sec_final_discussion}, we show that in general the Lower Way strategies do not encompass  all optimal strategies. 
However for a particular class of graphs that we call 
\emph{iteratively twin-free},   if the robber plays a Higher Way strategy, then all optimal strategies for the cop are Lower Way strategies.
We raise the question of ``characterizing all the optimal strategies.''

Many of the results in this paper, though obtained independently, are explicitly or implicitly contained in a paper of Clarke, Finbow, and MacGillivary \cite{CNG14}.  
The fundamental difference between our papers is that all our results follow from
the properties of the Lower Way and Higher Way strategies, while their results
follow from a combination of the following tools:
1) A cop strategy, 2) inductive arguments on the corner rank, and 3) the well-known
characterization of cop-win graphs by dismantling orderings.  We do not claim
that our approach is better.  However, our approach does have the advantage
of providing explicit optimal strategies for both players, and we prove some interesting properties of these strategies.  In addition, in Section~\ref{sec_threeCors} we use our tools to
give an alternate proof of the characterization of the cop-win graphs as exactly those that have a dismantling ordering \cite{NW83, Qui78}.
At various points throughout the paper, we will compare our approach with that of \cite{CNG14}.  

The notion of corner rank provides a useful tool for investigating cop-win graphs.  We already see the usefulness in this paper.  We also have a submitted paper \cite{OO16} in which we use corner rank to characterize graphs of a given size with maximal capture time.  We expect corner rank to continue to be a useful tool for future questions about cop-win graphs.

\section{Corner Ranking} \label{sec_corner_ranking}

One of the fundamental results about the game of Cops and Robber is a characterization of the cop-win graphs as those graphs which have a \emph{dismantling ordering} (also known as a \emph{cop-win ordering}--see \cite{NW83} and \cite{Qui78}).  
In this section we define corner ranking, a procedure that can be thought of as a modification of the cop-win ordering.  In \cite{CNG14}, the authors describe \emph{cop-win partitions}, which are equivalent to corner ranking when restricted to cop-win graphs; at the end of this section we compare our approaches.

In this paper all graphs are finite and non-empty, i.e. have at least one vertex; all numbers are integers.
We follow a typical Cops and Robber convention: we assume that all graphs are reflexive, that is \emph{all graphs have a loop at every 
vertex so that a vertex is adjacent to itself}; we will never draw or mention such edges. This assumption simplifies much of the following discussion while leaving the game play unchanged. 
For a graph \graph{G}, $V(\graph{G})$ refers to the vertices of \graph{G}.
If \graph{G} is a graph and $X$ is a vertex or set of vertices in \graph{G},
then by $\graph{G} - X$ we mean the subgraph of \graph{G} induced by $V(\graph{G})\setminus X$. We say that a vertex $u$ \dword{dominates} a set of vertices $X$ if 
$u$ is adjacent to every vertex in $X$.
Given a vertex $v$ in a graph, by \clnb{v}, the \dword{closed neighborhood of v}, we mean the set of vertices adjacent to $v$. 
For a vertex $v$ in \graph{G} and an induced  subgraph \graph{H} of \graph{G},
let $\graph{H} \cup v$ be the subgraph of \graph{G} induced by
$V(\graph{H}) \cup \{ v \}$.
For distinct vertices $v$ and $w$, if $\clnb{v} \subseteq \clnb{w}$ then
we say that $v$ is a \dword{corner} and that $w$ \dword{corners} $v$.

\begin{convention}
We will commonly make reference to ``the cop'' and ``the robber'', to
refer to the vertices at which they are located; for example, to say that
the cop corners the robber, means that the cop is located at some 
vertex $c$, the robber at some vertex $r$, and $c$ corners $r$.
\end{convention}

It is well known that  for a graph to be cop-win, it must contain a corner, and furthermore,  if $v$ is a corner in a graph \graph{G}, then  \graph{G} is cop-win if and only if $\graph{G} - v$ is cop-win. 
Thus a graph is cop-win if and only if the vertices can be deleted one by one, where each vertex deleted is a corner at the time of deletion. The ordered list of these vertices is called a \dword{cop-win ordering} or 
\dword{dismantling ordering} \cite{NW83, Qui78}; 
for example, in the graph of Figure~\ref{cranex1}, the following is a dismantling ordering, where we delete vertices from left to right:
$$  a, d, e, g, h, f, c, b.$$
There may be many cop-win orderings for a given graph.
To build upon this work, we classify corners into two categories, which we call 
\emph{strict corners} and \emph{twins}. 

\begin{defn}
Suppose $v$ and $w$ are distinct vertices where $w$ corners $v$. 
\begin{itemize}
\item If $\clnb{v} \subsetneq \clnb{w}$, we say that $v$ is a \dword{strict corner} and that $w$ \dword{strictly corners} $v$.
\item If $\clnb{v} = \clnb{w}$, we call $v$ and $w$ \dword{twins}.

\end{itemize}
\end{defn}
\noindent
Note that if $v$ and $w$ are twins then $v$ corners $w$ and $w$ corners $v$.  If a vertex is a corner, it must either be a strict corner or a twin with some other vertex; it can also be both.
For example, in Figure~\ref{cranex1}, vertex $a$ is a strict corner but not a twin 
with any vertex.  Vertices $c$, $d$, and $e$ are pair-wise twins, but not strict corners. Vertices $g$ and $h$ are both strict corners and twins with one another.

\begin{defn}
In a graph, a \dword{path of strict corners} is a path $(v_1, \ldots, v_k)$ containing at least one vertex such 
that each $v_i$ is a strict corner and $v_{i+1}$ strictly corners $v_i$.
\end{defn}

\noindent
If $v$ is a strict corner, then the trivial path $(v)$ counts as a path of strict corners.
Also note that if $x$ is strictly cornered by $y$ and $y$ is strictly cornered by 
$z$, then $x$ is strictly cornered by $z$; thus in a path of strict corners $(v_1, \ldots, v_k)$,
if $i < j$ then $v_i$ is strictly cornered by $v_j$.

\begin{lemma}\label{noncorndom}
In a graph, any strict corner $v$  is strictly cornered by some vertex $w$, where $w$ is \emph{not} a strict corner.
\end{lemma}

\begin{proof}
Let $v$ be a strict corner and let $(v_1, \ldots, v_k)$ be a maximal path of strict corners
such that $v_1 = v$.  As observed above, $v_k$ strictly corners $v_i$ if $i < k$. Thus for all $i<k$, $v_k$ has at least one neighbor not adjacent to $v_i$, and so $v_i$ does not corner $v_k$.  
Since $v_k$ is a strict corner, it must be strictly cornered by some vertex $w \notin \{v_1, \ldots, v_k \}$.  By the maximality of the path, $w$ is \emph{not} a strict corner.
Since $v$ is strictly cornered by $v_k$ and $v_k$ is strictly cornered by $w$, $v$ is strictly cornered by $w$, which is not a strict corner.
\end{proof}

We now define the corning ranking procedure, which can be thought of as an alternative to cop-win ordering where all strict corners are deleted at once in each step. 
Every successive group of strict corners is assigned the next positive integer rank.  The
procedure stops when it arrives at a graph with no strict corners.  If it terminates with a clique, the vertices are assigned the next positive integer; otherwise the remaining vertices are assigned the value $\infty$.

\begin{defn} [{\bf Corner Ranking Procedure}]
\label{def_corner_ranking}

For any graph \graph{G},
we define a corresponding  \dword{corner rank} function, $\mathbf{cr}$, which maps each vertex of \graph{G} to a positive integer or $\infty$.
We also define a sequence of associated graphs 
$\uprank{\graph{G}}{1}, \ldots, \uprank{\graph{G}}{\alpha}$.

\begin{enumerate}
\setcounter{enumi}{-1}

\item Initialize $\uprank{\graph{G}}{1} =\graph{G}$, and $k = 1$.

\item 
\label{alg_start}
If \uprank{\graph{G}}{k} is a clique, then:
\begin{itemize}

\item
 Let $\rfunct{x} = k$ for all $x \in  \uprank{\graph{G}}{k}$.

\item
Then  
{\bf stop}. 

\end{itemize}

\item Else if \uprank{\graph{G}}{k}
is \emph{not} a clique and has {\bf no} strict corners,  then:
\begin{itemize}

\item
Let $\rfunct{x} = \infty$ for all $x \in \uprank{\graph{G}}{k}$.

\item 
Then {\bf stop}. 

\end{itemize}

\item Else:  
\begin{itemize}
\item Let $\X$ be the set of strict corners in \uprank{\graph{G}}{k}.

\item For all $x \in \X$, let  $\rfunct{x} = k$.  
\item Let 
$\uprank{\graph{G}}{k+1} = \uprank{\graph{G}}{k} - \X$

\item Increment $k$ by 1 and return to Step~\ref{alg_start}.
\end{itemize}
\end{enumerate}
Define the \dword{corner rank} of \graph{G}, denoted \rfunct{\graph{G}}, to be the same as the vertex of \graph{G}
with largest corner rank; we understand $\infty$ to be larger than all integers.

\end{defn}

\noindent
Lemma~\ref{noncorndom} implies that every graph that contains a strict corner also contains at least one vertex that is \emph{not} a strict corner, so every graph \uprank{\graph{G}}{k}
produced in the corner ranking procedure will be nonempty.

As an example, we apply the corner ranking procedure to the graph in Figure~\ref{cranex1}.
The procedure begins by assigning the 
strict corners $a$, $g$, and $h$ corner rank 1.
After these vertices are removed, $b$ and $f$ are strict corners in 
\uprank{\graph{G}}{2},
and are thus assigned corner rank 2.
After these are removed, the remaining vertices, $c$, $d$, and $e$,
form a clique and so are assigned corner rank 3; thus the corner rank of the graph is 3.   
As another example, consider Figure~\ref{cranex2}. While $\rfunct{x} = 1$  and $\rfunct{y} = 2$, once $x$ and $y$ have been removed there are 
no strict corners, and what remains is not a clique,
so the  other 5 vertices have corner rank $\infty$;  thus the graph
has corner rank $\infty$. We will prove that a graph is cop-win if and only if it has a finite corner rank.  For example, the graph in Figure~\ref{cranex1} has finite rank and is cop-win, while the graph in Figure~\ref{cranex2} has infinite rank and is not cop-win.
\emph{Unless we explicitly state that a graph has finite corner rank, we assume that its corner rank may be infinite.}

\begin{figure}
\begin{center}
  \begin{tikzpicture}[every node/.style={circle,fill=black!20}]  
      \node (n1) at (0,0) [label=above:$a$] {$1$};
      \node (n2) at (1,0) [label=above:$b$] {$2$};
      \node (n31) at (3,-1) [label=left:$e$] {$3$};
      \node (n32) at (2,0) [label=above:$c$] {$3$};
      \node (n33) at (3,1) [label=above:$d$] {$3$};
      \node (n4) at (4,0) [label=above:$f$]  {$2$};
      \node (n51) at (5,-.5) [label=right:$h$]  {$1$};
      \node (n52) at (5,.5) [label=right:$g$] {$1$};
    \foreach \from/\to in {n1/n2, n2/n31, n2/n32, n2/n33, n31/n32, n32/n33, n31/n33, n31/n4, n32/n4, n33/n4, n4/n51, n4/n52, n51/n52}
    \draw (\from) -- (\to);
    \end{tikzpicture}
\end{center}
\caption{The corner ranking of a cop-win graph.}\label{cranex1}
\end{figure}

\begin{figure}
\begin{center}
  \begin{tikzpicture}[every node/.style={circle,fill=black!20}]  
      \node (n3) at (0,0)  {$\infty$};
      \node (n4) at (0,2)  {$\infty$};
      \node (n5) at (2,2)  {$\infty$};
      \node (n2) at (2,0)  {$\infty$};
      \node (n1) at (4,1)  {$\infty$};
      \node (y) at (6,1) [label=above:$y$]  {$2$};
      \node (x) at (8,1) [label=above:$x$]  {$1$};
    \foreach \from/\to in {n1/n2, n2/n3, n3/n4, n4/n5, n5/n1, n1/y, y/x}
    \draw (\from) -- (\to);
    \end{tikzpicture}
\end{center}
\caption{The corner ranking of a non cop-win graph.}\label{cranex2}
\end{figure}

A graph with rank one is either a single vertex, with capture time 0, or a clique on two or more vertices, with capture time 1.
Since the optimal strategy and capture time are clear in these cases, for the rest of the paper we make the following convention.
\begin{convention}
We assume all finite rank graphs have corner rank at least 2.
\end{convention}

For cop-win graphs, we will see that the strategy and capture time depend on a structural  
property of the highest ranked vertices.

\begin{defn} \label{defn_topheavy}
Suppose \graph{G} is a graph with a \emph{finite} corner rank $\alpha$.
We say that \graph{G} is a \dword{$1$-cop-win graph}
if one of the two equivalent conditions holds:

\begin{itemize}
\item 
Some vertex of corner rank $\alpha$ dominates 
$V(\uprank{\graph{G}}{\alpha - 1})$.

\item 
Every vertex of corner rank $\alpha$ dominates 
$V(\uprank{\graph{G}}{\alpha - 1})$.

\end{itemize}
\noindent
Otherwise we say \graph{G} is a  \dword{$0$-cop-win graph}.

\end{defn}

\noindent
For example, the graph in Figure~\ref{cranex1} is $1$-cop-win,
since any of the three vertices of rank 3 ($e,d,$ and $c$) are adjacent to all the rank 3 vertices and all the vertices of rank 2 ($b$ and $f$).  The graph in Figure~\ref{lowmid} is $0$-cop-win.
\emph{We remark that at this point we have not proven the being $r$-cop-win for some $r$ implies being cop-win; this will be shown in Corollary~\ref{cor_capture_upper}, thus justifying the names for our terminology.}
We now show that the two conditions in Definition~\ref{defn_topheavy} are actually equivalent.
\begin{lemma} \label{allvertsxa}
For a graph with finite corner rank $\alpha$, 
the two conditions in Definition~\ref{defn_topheavy} are equivalent.
\end{lemma}
\begin{proof}

Let \graph{G} be a graph with finite corner rank $\alpha$.  
It suffices to show that if there is some vertex $x$ of corner rank $\alpha$ dominating 
$V(\uprank{\graph{G}}{\alpha-1})$,
then every vertex of rank $\alpha$ dominates
$V(\uprank{\graph{G}}{\alpha - 1})$. 
Assume for contradiction that $y$ is a vertex of rank $\alpha$ that does not dominate 
$V(\uprank{\graph{G}}{\alpha - 1})$.
Then $x$ strictly corners $y$ in 
\uprank{\graph{G}}{\alpha - 1}, 
contradicting the assumption that $\rfunct{y}=\alpha$.
\end{proof}

In \cite{CNG14}, corner ranking is also defined for cop-win graphs; they call it a \emph{cop-win partition}.
One difference is that we define vertices with corner rank $\infty$.
In Section~\ref{OneStep}, we show that the robber's Higher Way strategy allows him to safely remain among the corner rank $\infty$ vertices indefinitely.  Another difference is that 
our procedure appears to be simpler.
We simply remove all the strict corners at each step.  In \cite{CNG14}, they essentially carry out two actions in each step: First they collapse any vertices that are currently twins to a single vertex, and then they remove all the corners.

\section{Projections}
\label{sec_projections}

A key technical point will be to deal with projections of cop or robber locations rather than their actual locations. First,
we introduce some definitions that make our treatment of finite and infinite corner rank graphs
uniform.

\begin{defn}
Suppose we are considering a graph with corner rank $\infty$.

\begin{itemize}

\item For any finite corner rank $k$ exhibited by this graph, let $k < \infty$.

\item
If there is a vertex of finite corner rank, then there is a largest finite corner rank, say $\gamma$.
In this case, we let $\infty - 1$ refer to $\gamma$
and let $\gamma + 1$ refer to $\infty$.

\end{itemize}

\end{defn}

\noindent
For example, if a graph has vertices with corner ranks $1,2,$ and $\infty$, and nothing else, then
$1, 2 < \infty$.  Furthermore, $2+1 = \infty$ and $\infty -1 = 2$.

Now we define the projection functions, relative to the corner ranking. A similar definition appears in \cite{CNG14}, which is in turn a modification of definitions
from \cite{CN2001}.
We write $f:\graph{H} \multfunct \graph{G}$ to mean that $f$ is a function whose domain is the {\bf non-empty subsets} of $V(\graph{H})$ and whose codomain is the {\bf non-empty subsets} of $V(\graph{G})$.


\begin{defn}
Suppose \graph{G} is a graph with corner rank $\alpha$.
We define the functions $f_1, \ldots, f_{\alpha -1}$ and 
$F_1, \ldots, F_{\alpha-1}, F_{\alpha}$,
where
$f_k: \uprank{\graph{G}}{k}  \multfunct \uprank{\graph{G}}{k+1}$ and
 $F_k:\graph{G} \multfunct \uprank{\graph{G}}{k}$.

\begin{itemize}

\item
For a single vertex $u \in V(\uprank{\graph{G}}{k})$, define:
$$f_k(\{u\}) = 
\begin{cases}
\{u\} & \hbox{if $\rfunct{u} > k$} \\
\hbox{the set of vertices in \uprank{\graph{G}}{k+1} 
that strictly corner $u$ in \uprank{\graph{G}}{k} } & \hbox{\bf otherwise.}
\end{cases}
$$

\item $f_k(\{u_1, \ldots, u_t\}) = \bigcup\limits_{1\le i \le t}f_k(u_i)$

\item If \graph{G} has some finite rank vertex, then: 
\begin{itemize}

\item
Let $F_1:  \graph{G} \multfunct  \graph{G}$ be the identity function,  and

\item
For $1 < k \le \infty$, let $F_k = f_{k-1} \circ \cdots \circ f_1$.

\end{itemize}

\item If \graph{G} has no finite rank vertices, then
let $F_{\infty}:  \graph{G} \multfunct  \graph{G}$ be the identity function.

\end{itemize}

\end{defn}
\noindent
For a function $h$ whose domain is sets of vertices, we adopt the usual convention that $h(u) = h(\{ u \})$ for a single vertex $u$. We remark that by
by Lemma~\ref{noncorndom} the functions $f_k$ are guaranteed to have
\emph{non-empty} sets for values.
We say $v$ is a \dword{$k$-projection} (or simply a \dword{projection}) of $w$ if $v \in F_k(w)$. 

\begin{defn}
Given two graphs \graph{H} and \graph{G}, and a function $h: V(\graph{H}) \multfunct V(\graph{G})$, we say that $h$ is a \dword{homomorphism} if for vertices $u, v \in V(\graph{H})$ and vertices $u^* \in h(u), v^* \in h(v)$:
\[
u \hbox{ is adjacent to } v \ \hbox{implies} \ u^* \hbox{ is adjacent to } v^*.
\]

\end{defn}

\begin{lemma}  
Given a graph \graph{G}, its  associated functions $f_k$ and $F_k$ 
are homomorphisms.
\end{lemma}

\begin{proof}
Since the identity function is a homomorphism, for a graph with some vertices of finite corner rank, $F_1$ is a homomorphism, as is $F_\infty$ on graphs with no vertices of finite corner rank. We show that each $f_k$ is a homomorphism, which implies all other $F_k$'s are homomorphisms
because the property is preserved by composition.
Suppose $u, v \in V(\uprank{\graph{G}}{k})$ are distinct and adjacent, and let $u^* \in f_k(u)$ and
$v^* \in f_k(v)$; we show that $u^*$ and $v^*$ are adjacent.  Note that even if $u^* = v^*$, the argument works since our graphs are reflexive; in particular, to say that $w$ corners $v$, simply amounts to saying that $w$ is adjacent to everything $v$ is adjacent to (this includes $w$ being adjacent to itself and to $v$).

\textbf{Case 1:}  Suppose $u,v \in V(\uprank{\graph{G}}{k+1})$.  Then $f_k(v) = \{ v \}$ and $f_k(u) = \{ u \}$,
and so $u^* = u$ and $v^* = v$ are  adjacent.

\textbf{Case 2:} Suppose  $v \in V(\uprank{\graph{G}}{k+1})$, $u \notin V(\uprank{\graph{G}}{k+1})$.
So $v^* = v$ and $u^* \neq u$.  
Since $u^*$ strictly corners $u$ in \uprank{\graph{G}}{k}, $u^*$ is adjacent to $v$, and thus $u^*$  and $v^*$ are adjacent.

\textbf{Case 3:} Suppose $u,v \notin V(\uprank{\graph{G}}{k+1})$. 
So $v^* \neq v$ and $u^* \neq u$.  
Since $u^*$ strictly corners $u$ in \uprank{\graph{G}}{k}, $u^*$ is adjacent to both $u$ and $v$. Since $v^*$ strictly corners $v$ in
\uprank{\graph{G}}{k}, and $v$ is adjacent to $u^*$, we have that
$v^*$ is also adjacent to $u^*$. 
\end{proof}

We state an important fact that follows from the properties of the projection functions.

\begin{defn}
For a vertex $v$ in a graph, let $\clnbk{k}[v]$ denote the set of vertices in $\clnb{v}$ of 
rank at least $k$.
\end{defn}

\begin{lemma}\label{projnbr}
Let $c$ be a vertex in a graph, and $c'\in F_k(c)$.  Then $\clnbk{k}[c] \subseteq \clnbk{k}[c']$.
\end{lemma}

\begin{proof}
Consider a vertex $w \in \clnbk{k}[c]$.  Since $F_k$ is a homomorphism, any vertex in $F_k(w)$ is adjacent to $c'$, but since $w$ has rank at least $k$, $F_k(w) = \{w\}$, so $w\in \clnbk{k}[c']$.
\end{proof}

\section{Cop Strategy: Lower Way}\label{LowWay}
We describe a class of winning strategies for the cop, which allows us to prove that any graph with a finite corner rank is cop-win, and yields an upper bound on the capture time.
These strategies will turn out to be optimal for the cop and will be called \emph{Lower Way} since they
(roughly) have the property that as the game proceeds, the cop forces the robber to occupy lower-ranked vertices.



\begin{defn} Suppose $c$ and $r$ are vertices in a graph \graph{G}.
We say that $r$ is \dword{$0$-cornered} by $c$ if $c=r$.  For $k>0$, we say that $r$ is \dword{$\mathbf{k-}$cornered} by $c$ if there is some  $r' \in F_k(r)$ such that
$c$ corners $r'$ in $\uprank{\graph{G}}{k} \cup c$.

\end{defn}

\noindent
We do \emph{not} require strict cornering in this definition.
Note that $1$-cornered means the same thing as cornered.

\begin{defn} Consider a graph with vertices $c$ and $r$, and an integer $k \ge 1$.
We say that $c$ has \dword{$\mathbf{k-}$caught} $r$ if $c \in F_k(r)$.

\end{defn}

\noindent
Notice that $1$-caught means the same thing as caught.

When a cop has actually cornered the robber, then whatever the robber does, the cop can catch him on the next move.
The next lemma is a version of that fact for projections.

\begin{lemma} \label{lemma_cornertocatch}
For all integers $k \ge 1$,
if the cop has $k$-cornered the robber, then whatever move the robber makes, 
the cop can $k$-catch the robber on her next move.

\end{lemma}

\begin{proof} Given a graph \graph{G}, assume the cop is at some vertex $c$ and the robber is at some vertex $r_0$ such that for some $r_0' \in F_k(r_0)$, $c$ corners $r_0'$ in 
$\uprank{\graph{G}}{k} \cup c$.
Now suppose the robber moves from $r_0$ to $r_1$; let $r_1' \in F_k(r_1)$. 
Since $r_0$ and $r_1$ are adjacent and $F_k$ is a homomorphism, $r_0'$ and $r_1'$ must also be adjacent.  Since $c$ corners $r_0'$ in 
$\uprank{\graph{G}}{k} \cup c$, $c$
must be adjacent to $r_1'$, so the cop can move to $r_1'$ and thus $k$-catch the robber.
\end{proof}

\begin{lemma} \label{lemmaCaughtToCorner}
For all integers $k \ge 1$,
if the cop has $k$-caught the robber, then the cop has $(k-1)$-cornered the robber. 
\end{lemma}

\begin{proof}
If $k=1$, then the cop has caught the robber and thus the robber is 0-cornered. Assume $k>1$.
Suppose the cop is at $c$ and the robber is at $r$.  Since $c \in F_k(r) =  f_{k-1} \circ  F_{k-1} (r)$, either $c \in F_{k-1}(r)$ or 
there is an $r' \in F_{k-1}(r)$ such that $c$ strictly corners $r'$ in 
\uprank{\graph{G}}{k-1}.
In either case, the robber is $(k-1)$-cornered.
\end{proof}

Lemmas~\ref{lemma_cornertocatch} and \ref{lemmaCaughtToCorner} yield the following corollaries.
\begin{cor} \label{cor_corner_reduction}
For all integers $k \ge 1$, if the cop has $k$-cornered the robber, then whatever move the robber makes, 
the cop can $(k-1)$-corner the robber on her next move.
\end{cor}

\begin{cor} \label{cor_catch_reduction}
For all integers $k \ge 2$, if the cop has $k$-caught the robber, then whatever move the robber makes, 
the cop can $(k-1)$-catch the robber on her next move.
\end{cor}

To define Lower Way strategies, we need to describe the starting positions for the cop.
\begin{defn}\label{stansta}
 Consider an $r$-cop-win graph \graph{G} of finite corner rank $\alpha$.
An initial placement by the cop is called \dword{standard} if that vertex 
dominates $V(\uprank{\graph{G}}{\alpha - r})$.

\end{defn}

\noindent
In both $0$-cop-win and $1$-cop-win graphs, any vertex of rank $\alpha$ is a standard initial placement, though there may be more. Furthermore, Lemmas~\ref{lemma_robber_start} and \ref{projnbr} imply: in a 
$0$-cop-win graph, while the cop can dominate 
$V(\uprank{\graph{G}}{\alpha})$ 
(i.e. the top corner rank vertices), the cop cannot dominate 
$V(\uprank{\graph{G}}{\alpha - 1})$ 
(i.e. the top two
corner rank vertices), and in a 
$1$-cop-win graph, while the cop can dominate 
$V(\uprank{\graph{G}}{\alpha - 1})$, the cop cannot dominate
$V(\uprank{\graph{G}}{\alpha - 2})$
 (i.e. the top three corner rank vertices).

\begin{defn} [{\bf Lower Way and Catching Strategies}]
\label{def_lw_catching}

Consider an $r$-cop-win graph with finite corner rank $\alpha$. 

\begin{itemize}

\item
A cop strategy is called \dword{Lower Way} if the cop's initial placement is standard, and it satisfies the condition LW.
\begin{quote}
{\bf LW}: For $t\ge 1$, after $t$ cop moves, the cop $k$-corners the robber for some
$k \le \alpha - r - t$.
\end{quote}

\item
A cop strategy is called \dword{Catching} if the cop's initial placement is standard, and it satisfies the condition C.
\begin{quote}
{\bf C}: 
For $t\ge 1$, after $t$ cop moves, the cop has $k$-caught the robber for some
$k \le \alpha - r - t + 1$.
\end{quote}

\end{itemize}

\end{defn}

\noindent
A simple special case of the Lower Way strategy is as above,
with condition LW replaced by:
\emph{The cop moves to $k$-corner the robber for $k$ minimum.} 
Similarly, a simple special case of the Catching strategy replaces condition C by:
\emph{The cop moves to $k$-catch the robber for $k$ minimum.} The Catching strategies are basically described in \cite{CNG14} (technically the Catching strategies subsume their strategies, since in \cite{CNG14} they simply have the cop start in 
\uprank{\graph{G}}{\alpha}).

The next two theorems follow immediately from repeated application of Corollary~\ref{cor_corner_reduction} or
Corollary~\ref{cor_catch_reduction}.
\begin{thm} If the cop has $k$-cornered the robber or
$(k+1)$-caught the robber, then the cop needs at most $k$ more moves to win.
\end{thm}

\begin{thm} \label{thm_LW_C_continuation} \

\begin{itemize}

\item
If the cop ever satisfies the {\bf LW} condition, then she can continue to satisfy the {\bf LW} condition for the rest of the game.

\item
If the cop ever satisfies the {\bf C} condition, then she can continue to satisfy the {\bf C} condition for the rest of the game.

\end{itemize}

\end{thm}

\begin{lemma}\label{catisl}  
Any Catching strategy is also a Lower Way strategy.
\end{lemma}

\begin{proof}
Suppose the cop plays a Catching strategy. Then she will start in a standard position, consistent with a Lower Way strategy.  
For all $t\ge 1$, after $t$ cop moves, the cop has $k$-caught the robber for some $k \le \alpha - r - t + 1$. Thus by Lemma~\ref{lemmaCaughtToCorner}, for all $t\ge 1$, after $t$ cop moves, the cop has $k$-cornered the robber for some $k \le \alpha - r - t$, and thus the strategy is also a Lower Way strategy.
\end{proof}

\noindent
We will discuss some of the differences between Lower Way and Catching strategies in Section~\ref{sec_final_discussion}. In particular, in Theorem~\ref{catnlw} we show that there are Lower Way strategies that are not Catching strategies.

\begin{thm} \label{thm_cop_strat}
 Consider an $r$-cop-win graph of finite corner rank $\alpha$.
There exist Lower Way and Catching strategies.  Furthermore, if the
cop follows a Lower Way strategy or a Catching strategy, then the cop will win in at most $\alpha - r$ moves.

\end{thm}
\begin{proof}

By Lemma~\ref{catisl} to show that a Lower Way strategy exists, it suffices to show that a Catching strategy exists. We choose to have
the cop start at a vertex of corner rank $\alpha$; this is a standard start.
If the graph is $0$-cop-win, then since the cop dominates 
$V(\uprank{\graph{G}}{\alpha})$, no matter where the robber's initial placement is, on her first move, she can $\alpha$-catch him.   If the graph is $1$-cop-win,  then since the cop dominates 
$V(\uprank{\graph{G}}{\alpha - 1})$, no matter where the robber's initial placement is, on her first move, she can $(\alpha-1)$-catch him.  Either way, after one move, the cop satisfies the {\bf C} condition; thus by Theorem~\ref{thm_LW_C_continuation}, the cop can continue to satisfy the {\bf C} condition.  So Catching strategies exist, and thus so do Lower Way strategies.

The upper bounds follow immediately from the definitions of the strategies.
If the cop plays a Lower Way strategy, the cop will $0$-corner (and thus catch) the robber after $\alpha -r$ cop moves.  For a Catching strategy, the cop will $1$-catch (and thus catch) the robber after $\alpha - r$ cop moves.
\end{proof}

\noindent
Thus we can derive an important upper bound and justify our use of the terminology $r$-cop-win as only referring to cop-win graphs.
\begin{cor} \label{cor_capture_upper}
Suppose \graph{G} is an $r$-cop-win graph with finite corner rank $\alpha$.  Then \graph{G} is cop-win and 
$\capt(\graph{G}) \le \alpha - r$.

\end{cor}

\section{Robber Strategy: Higher Way}\label{OneStep}

In this section we will prove lower bounds that match the upper bounds from Corollary~\ref{cor_capture_upper}.  Just as the upper bound of the last section was a result of an explicit class of cop strategies, the lower bound of this section will follow from an explicit class of robber strategies we call the \emph{Higher Way} strategies.
At the end of the section we will contrast our approach to that of \cite{CNG14}.

\begin{defn} Suppose the cop is at vertex $c$ in a graph, and 
$k \ge 1$. 
\begin{itemize}

\item We say that vertex $r$ is $k$-\dword{safe} if $\rfunct{r} \ge k$ and $c$ is \emph{not} adjacent to $r$.

\item We say that vertex $r$ is $k$-\dword{proj-safe} if $\rfunct{r} \ge k$ and
there is $c' \in F_k(c)$ such that $c'$ is \emph{not} adjacent to $r$.

\end{itemize}
\end{defn}

\noindent
Roughly, a $k$-proj-safe robber is safe from the $k$-projection of the cop, while a $k$-safe robber is safe from the cop herself.  Lemma~\ref{projnbr} implies that if $r$ is $k$-proj-safe, then $r$ is $k$-safe;  however the converse need not be true.

\begin{defn}
We say that a robber strategy is \dword{Higher Way} if at every turn (including the initial placement), the robber goes to a $k$-proj-safe vertex for $k$ maximum; if there is no such option, he moves arbitrarily to a vertex not occupied by the cop.
\end{defn}

\noindent
By our convention, any graph has rank at least two, and so has more than 1 vertex.  For a graph with more than one vertex it is immediate that Higher Way strategies exist, because the robber can start at an unoccupied vertex and during play can always stay still if there is no other option.

The next lemma will be used by the robber to start safely at a sufficiently high rank vertex.

\begin{lemma} \label{lemma_robber_start}
Suppose a graph \graph{G} has  rank $\alpha$; the cop can be anywhere.
\begin{enumerate}

\item If the graph is $0$-cop-win, then there is a $(\alpha-1)$-proj-safe vertex. 

\item  If the graph  is $1$-cop-win and $\alpha > 2$, then there is a $(\alpha-2)$-proj-safe vertex.

\item  If the graph has rank $\infty$, then there is an $\infty$-proj-safe vertex.

\end{enumerate}
\end{lemma}

\begin{proof}

Suppose the cop is at vertex $c$. It suffices to show that in Case 1 (resp. Case 2, Case 3) there is no vertex in 
\uprank{\graph{G}}{\alpha - 1}
(resp. 
\uprank{\graph{G}}{\alpha - 2}, \uprank{\graph{G}}{\infty})
that dominates 
$V(\uprank{\graph{G}}{\alpha - 1})$
(resp. 
$V(\uprank{\graph{G}}{\alpha - 2})$, $V(\uprank{\graph{G}}{\infty})$).

\begin{enumerate}

\item
Suppose \graph{G} is $0$-cop-win, and for the sake of contradiction, assume some vertex in 
\uprank{\graph{G}}{\alpha - 1}
dominates 
$V(\uprank{\graph{G}}{\alpha - 1})$.
Such a vertex could not be strictly cornered in 
\uprank{\graph{G}}{\alpha - 1}
so it would have rank $\alpha$, making \graph{G} $1$-cop-win, a contradiction.  

\item
Suppose \graph{G} is $1$-cop-win,  and for the sake of contradiction, assume some vertex $v$ in 
\uprank{\graph{G}}{\alpha - 2}
dominates 
$V(\uprank{\graph{G}}{\alpha - 2})$.
Any vertex not strictly cornered by $v$ in 
\uprank{\graph{G}}{\alpha - 2}
must be a twin of $v$ in 
\uprank{\graph{G}}{\alpha - 2}.

In 
\uprank{\graph{G}}{\alpha - 2}, $v$ cannot strictly corner any vertex in 
\uprank{\graph{G}}{\alpha - 1}, so all the vertices of 
\uprank{\graph{G}}{\alpha - 1}
are twins with $v$, and so they are twins with one another; i.e. 
\uprank{\graph{G}}{\alpha - 1}
is a clique. But 
\uprank{\graph{G}}{\alpha - 1}
being a clique contradicts the definition of corner rank. 

\item
Suppose $\graph{G}$ has rank $\infty$, and for the sake of contradiction, assume some vertex $v$ in 
\uprank{\graph{G}}{\infty}
dominates 
$V(\uprank{\graph{G}}{\infty})$.
Since none of the other vertices in 
\uprank{\graph{G}}{\infty}
are strictly cornered by $v$, all vertices in 
\uprank{\graph{G}}{\infty}
must be twins. Hence 
\uprank{\graph{G}}{\infty} is a clique, contradicting the definition of corner rank. 
\end{enumerate}
\end{proof}

The next lemma  shows how the robber can avoid being adjacent to a projection of the cop (and hence not adjacent to the cop) while moving down at most one corner rank on each turn.


\begin{lemma} \label{lemma_robber_response}
Suppose the cop is at some vertex $c_0$ and the robber is at a $k$-proj-safe vertex for $k \ge 2$. Suppose the cop moves from $c_0$ to $c_1$.

\begin{enumerate}

\item
The robber has a move to a 
$(k-1)$-proj-safe vertex.

\item
Furthermore, if there is a $c_1' \in F_k(c_1)$ such that the robber is 
\emph{not} cornered by $c'$ in 
\uprank{\graph{G}}{k}, then the robber has a move to a $k$-proj-safe vertex.

\end{enumerate}

\end{lemma}

\begin{proof} 
For the first part, suppose the robber is at $r_0$.  Since $r_0$ is $k$-proj-safe from $c_0$, there exists $c'_0 \in F_k(c_0)$ such that $c'_0$ is not adjacent to $r_0$.
Then the cop moves to $c_1$.  
Assume for contradiction that from $r_0$ the robber does not have a move to a $(k-1)$-proj-safe vertex.
For $r_0$ not to have such a move, means that for any $r_1 \in \clnbk{k-1}[r_0]$, for all
$c_1' \in F_{k-1}(c_1)$, $c_1'$ is adjacent to $r_1$,
i.e. all $c_1' \in F_{k-1}(c_1)$ corner $r_0$ in 
\uprank{\graph{G}}{k-1}.
Consider any such $c_1'$ which corners $r_0$ in 
\uprank{\graph{G}}{k-1}.
Since $r_0$ has rank $\ge k$, this cornering cannot be strict, and thus $c_1'$ and $r_0$ are twins in 
\uprank{\graph{G}}{k-1}. Since $r_0$ has rank $\ge k$, $c_1'$ must also have rank $\ge k$.  This implies that $f_{k-1}(c_1') = \{c_1'\}$ and so $c_1' \in F_k(c_1) =  f_{k-1} \circ  F_{k-1} (c_1)$. However since $F_k$ is a homomorphism and $c_1$ is adjacent to $c_0$, $c_1'$ is adjacent to $c_0'$. Since $r_0$ is not adjacent to $c_0'$, this contradicts the fact that $c_1'$ and $r_0$ are twins in 
\uprank{\graph{G}}{k-1}.

For the second part, note that there is a vertex of rank at least $k$ adjacent to $r$ that is not adjacent to $c_1'$, giving the robber a $k$-proj-safe move.
\end{proof}

\begin{thm} \label{thm_higherway}
Suppose \graph{G} is a graph of corner rank $\alpha$ and the robber follows a Higher Way strategy.

\begin{enumerate}

\item \label{HWcase_finite}
If \graph{G} is $r$-cop-win, the game will last at least $\alpha-r$ cop moves.



\item
\label{HWcase_infinite}
If $\alpha = \infty$, then the cop will never catch the robber.

\end{enumerate}

\end{thm}

\begin{proof}
From Lemma~\ref{lemma_robber_start} we conclude that in Case~\ref{HWcase_finite}, if $r = 0$, the robber will start $(\alpha - 1)$-proj-safe, and if $r=1$, the robber will start $(\alpha - 2)$-proj-safe (unless \graph{G} has rank 2, in which case he will simply avoid
placing himself at the cop's location).  In Case~\ref{HWcase_infinite}, the robber will start $\infty$-safe.  We now apply Lemma~\ref{lemma_robber_response} repeatedly in all cases.

\begin{itemize}

\item[]
{\bf Case~\ref{HWcase_finite}.}

Suppose the robber is $k$-proj-safe for some $k \ge 2$.
Since the robber is $k$-proj-safe, the cop is not adjacent to the robber and
by Lemma~\ref{lemma_robber_response} whatever the cop
does, the robber can respond by moving to a vertex 
that is at worst $(k-1)$-proj-safe.
Lemma~\ref{lemma_robber_response} can be repeatedly applied until the robber is forced to move to a $1$-proj-safe vertex.
At this point, the cop is not adjacent to the robber, so it takes at least 2 more cop moves to catch the robber.

Now we calculate lower bounds on the number of cop moves.  In the 
$0$-cop-win case, the robber starts at rank 
at least $\alpha - 1$ and survives for at least $\alpha - 2$ cop moves before reaching rank 1; the additional 2 cop moves means
that at least $(\alpha - 2) + 2 = \alpha$ cop moves are needed. The $1$-cop-win case is similar for $\alpha >2$, but with the robber 
starting at a vertex of rank at least $\alpha - 2$, so we get at least $(\alpha - 3) + 2 = \alpha - 1$ cop moves.

If $\alpha = 2$ and the graph is $1$-cop-win, then the robber initially goes to a vertex not occupied by the cop and thus the cop will make at least one move, as required.

\item[]
{\bf Case~\ref{HWcase_infinite}.}

Suppose the cop's initial position is $c_0$, and the robber's initial position is at the $\infty$-proj-safe vertex $r_0$, i.e. $r_0$ is not adjacent to some $c_0' \in F_\infty(c_0)$.  Suppose the cop moves to $c_1$.
Assume for contradiction that there is no move from $r_0$ to a $\infty$-proj-safe vertex.  That
means that for any $r_1 \in \clnbk{\infty}[r_0]$, for all $c_1' \in F_{\infty}(c_1)$, $c_1'$ is adjacent to 
$r_1$, i.e. all $c_1' \in F_{\infty}(c_1)$ corner $r_0$ in 
\uprank{\graph{G}}{\infty}.

Since strict cornering is not possible in 
\uprank{\graph{G}}{\infty}, this means that $r_0$ is twins with all
$c_1' \in F_{\infty}(c_1)$.  But since any $c_1'$ is adjacent to $c_0'$ (because $c_0$ is adjacent to $c_1$, and $F_{\infty}$ is a homomorphism), the vertex $r_0$ is adjacent to $c_0'$, a contradiction.

\end{itemize}
\end{proof}

\begin{cor} \label{lemma_capture_lower}
If an $r$-cop-win graph has finite corner rank $\alpha$, then
$\capt(\graph{G})  \ge \alpha - r$.

\end{cor}

We compare our approach to that of \cite{CNG14}.  In \cite{CNG14}, the lower bound of Corollary~\ref{lemma_capture_lower}
is proven in a natural way: by induction on the corner rank.  Our proof relies on the Higher Way strategy, an explicit description of a robber strategy.  As a result of our approach we get some benefits:
We have an explicit robber strategy, which not only delays the game as long as possible on cop-win graphs, but wins on non cop-win graphs.

\section{Main Results}
\label{sec_threeCors}

We can now put together the facts from the previous two sections to prove Theorem~\ref{thm_rank_capt_time}, one of our main theorems (also appears in \cite{CNG14}).
We can then deduce that the Lower Way and Higher Way strategies are in fact optimal (in Theorem~\ref{lwhwopt}).
We also give an alternate proof of the standard theorem characterizing cop-win graphs by dismantling orderings \cite{NW83, Qui78}.
Unlike \cite{CNG14}, we do not use the dismantling ordering characterization in our development, so we can in fact derive the result as a corollary of our work.

\begin{thm} \label{thm_rank_capt_time}
A graph is cop-win if and only if it has finite corner rank.  Furthermore, for an $r$-cop-win graph of finite corner rank $\alpha$,
$\capt(\graph{G}) = \alpha - r$.

\end{thm}
\begin{proof}  
For the claim that a graph is cop-win if and only if it has finite corner rank, 
the backward direction is noted in Corollary~\ref{cor_capture_upper}, while the forward direction follows from
 Theorem~\ref{thm_higherway} (Part 3): if the graph does not have finite corner rank, the robber has a winning strategy.
Furthermore, when the corner rank is finite, 
the upper bound on capture time in Corollary~\ref{cor_capture_upper} matches the lower bound in
Corollary~\ref{lemma_capture_lower}, so these corollaries together determine
the exact capture time.
\end{proof}

Since the Lower Way and Higher Way strategies both achieve the capture time in Theorem~\ref{thm_rank_capt_time},  the following is an immediate corollary.

\begin{thm} \label{lwhwopt}
On a graph with finite corner rank, any Lower Way strategy is optimal for the cop, and any Higher Way strategy is optimal for the robber.
\end{thm}

We now develop some ideas that we use to show how the dismantling ordering characterization \cite{NW83, Qui78}  follows from our work.

\begin{defn}
Suppose \graph{G} is a graph and $\bar{u} = (u_1, \ldots, u_n)$ is an ordering of  its vertices.
\begin{itemize}

\item
By $\graph{G}^{(u_i)}_{\bar{u}}$ we mean the graph induced by $\{u_i, \ldots, u_n\}$.

\item
We write $u_i \prec u_j$ when $i < j$.

\item
The pair $(\graph{G}, \bar{u})$ is \dword{good} for $u_i$ if there is some $u_j$ such that $u_i \prec u_j$ and $u_j$ corners $u_i$ in
$\graph{G}^{(u_i)}_{\bar{u}}$.

\end{itemize}

\end{defn}

\noindent
Observe that to show that an ordering $\bar{u}$ is a dismantling ordering for \graph{G} amounts to showing that for each $i$ ranging from $1$ to $n-1$, $(\graph{G}, \bar{u})$ is good for $u_i$.
If the ordering $\bar{u}$ is apparent from context, we may leave off the subscript
$\bar{u}$.  We prove a small fact about dismantling orderings.

\begin{lemma} \label{lemma_dismantle_swap}
Suppose a graph \graph{G} has a dismantling ordering
$(v_1, \ldots, v_n)$ such that some $v_k$ corners $v_j$ in $\graph{G}^{(v_{j-1})}$, where
$1 < j < k$.
Then swapping the positions of $v_j$ and $v_{j-1}$ in $(v_1, \ldots, v_n)$ and leaving the 
rest of the vertices in order, is another dismantling ordering of \graph{G}.
\end{lemma}

\begin{proof}
Let $\bar{\alpha} = (v_1, \ldots, v_n)$, and let $\bar{\beta}$ refer to the new ordering: 
$$(v_1, \ldots, v_{j-2}, v_j, v_{j-1}, v_{j+1}, \ldots, v_k, \ldots, v_n).$$
For all $i$ except $j-1$ and $j$, $(\graph{G}, \bar{\beta})$ is good for $v_i$ since $\graph{G}_{\bar{\alpha}}^{(v_i)}$ is the same graph as
$\graph{G}_{\bar{\beta}}^{(v_i)}$,
and we are given that $\bar{\alpha}$ is a dismantling ordering.

\begin{itemize}

\item
To see that $(\graph{G}, \bar{\beta})$ is good for $v_j$,
notice that $\graph{G}_{\bar{\alpha}}^{(v_{j-1})}$ and $\graph{G}_{\bar{\beta}}^{(v_{j})}$ are the same graphs, so since $v_k$ corners $v_j$ in the $\graph{G}_{\bar{\alpha}}^{(v_{j-1})}$, $v_k$ still corners $v_j$ in $\graph{G}_{\bar{\beta}}^{(v_{j})}$.

\item
To see that $(\graph{G}, \bar{\beta})$ is good for $v_{j-1}$,
consider cases on what corners $v_{j-1}$ in $\graph{G}_{\bar{\alpha}}^{(v_{j-1})}$:
\begin{itemize}

\item \emph{Case}: $v_j$.

Since $v_k$ corners $v_j$ in  $\graph{G}_{\bar{\alpha}}^{(v_{j-1})}$, and $v_j$
corners $v_{j-1}$ in $\graph{G}_{\bar{\alpha}}^{(v_{j-1})}$,
$v_k$ corners $v_{j-1}$ in $\graph{G}_{\bar{\alpha}}^{(v_{j-1})}$, and thus 
$v_k$ corners $v_{j-1}$ in $\graph{G}_{\bar{\beta}}^{(v_{j-1})}$.

\item \emph{Case}: Otherwise.

If something other than $v_j$ cornered $v_{j-1}$ in $\graph{G}_{\bar{\alpha}}^{(v_{j-1})}$, then it still does in $\graph{G}_{\bar{\beta}}^{(v_{j-1})}$. 

\end{itemize}

\end{itemize}
\end{proof}

\begin{lemma} \label{lemma_rank_one_first}
Any graph with a dismantling ordering has a dismantling ordering
in which all its rank 1 vertices are to the left of any higher ranked vertices.
\end{lemma}

\begin{proof}
Suppose
\graph{G} is a graph, and $(v_1, \ldots, v_n)$ is some dismantling ordering.
If the rank 1 vertices are already left of any higher ranked vertices we are done,
so assume otherwise.
Let $v_j$ be the leftmost vertex of corner rank 1 such that $v_{j-1}$ has corner rank at least 2.
It suffices to show that there is a vertex $v_i \prec v_j$ such that $v_i$ has rank at least 2 and
swapping $v_i$ and $v_j$ yields a new dismantling ordering of \graph{G}; this suffices because we can repeatedly swap vertices until all the corner rank 1 vertices are on the left.
Since $v_j$ is corner rank 1, by Lemma~\ref{noncorndom} there is a vertex $w$ of corner rank at least 2 that corners $v_j$ in \graph{G}.
Consider cases:
\begin{enumerate}

\item
$v_j \prec w$:

Then $w$ corners $v_j$ in $\graph{G}^{(v_{j-1})}$.
Thus we apply Lemma~\ref{lemma_dismantle_swap} to swap $v_j$ and $v_{j-1}$.

\item
If $w \prec v_j$:

Then let $y$ be the rightmost vertex in the dismantling ordering with the following properties:
1) $y$ has corner rank at least 2,
2) $y \prec v_j$, and
3) $y$ corners $v_j$ in $\graph{G}^{(y)}$. 
Such a $y$ exists because $w$ is a vertex with these three properties.  

If $y$ and $v_j$ are twins in $\graph{G}^{(y)}$ then they can be swapped and the result is still
a dismantling ordering.
Otherwise $y$ strictly corners $v_j$ in $\graph{G}^{(y)}$.  Let $z$ be the vertex in $\graph{G}^{(y)}$ that corners $y$ in $\graph{G}^{(y)}$.  Thus $z$ strictly corners $v_j$ in $\graph{G}^{(y)}$, and thus also $z \neq v_j$. Note that since $v_j$ was the leftmost rank 1 vertex such that $v_{j-1}$ had rank at least 2, the vertices between $y$ and $v_j$ have rank at least 2.
Thus if $z$ is between $y$ and $v_j$ that would contradict the assumption that $y$ is the  rightmost vertex satisfying the three properties above.
Since $z \neq v_j$, we have that $v_j \prec z$.
Thus we apply Lemma~\ref{lemma_dismantle_swap} to swap $v_j$ and $v_{j-1}$.

\end{enumerate}
\end{proof}

\begin{lemma} \label{lemma_finiterank_dismantling}
A graph $\graph{G}$ has finite corner rank $\alpha$ if and only if it has a dismantling ordering.
\end{lemma}

\begin{proof}  \

\begin{itemize}

\item[]{\bf Forward Direction.} \
If $\graph{G}$ has finite corner rank then we claim that any list of the vertices where their ranks are non decreasing (i.e. the rank 1 vertices come first, in any order, followed by the rank 2 vertices in any order, and so on) is a dismantling ordering. Lemma~\ref{noncorndom} ensures that any vertex $v$ with rank less than $\alpha$ is cornered in $\graph{G}$ by a vertex of higher rank, and thus will be cornered by this vertex in $\graph{G}^{(v)}$. If there is more than one vertex of rank $\alpha$ then they all mutually corner each other, and thus any vertex $v$ of rank $\alpha$ is a corner in the clique $\graph{G}^{(v)}$.  Thus the sequence just described is in fact a dismantling ordering.

\item[]{\bf Backward Direction.} \
If all the vertices have corner rank 1 we are done,
so assume otherwise.
We proceed by
induction on the number of vertices in \graph{G}.
By Lemma~\ref{lemma_rank_one_first}, we can rearrange the given dismantling ordering into
a dismantling ordering that has all the corner rank 1 vertices first.  
By the definition of a dismantling ordering, removing any number of initial vertices from the ordering leaves a dismantling ordering of the remaining graph. Removing
all the initial corner rank 1 vertices results in a dismantling ordering of the smaller graph.  But this smaller graph is  just 
\uprank{\graph{G}}{2}, the graph induced by all vertices of \graph{G} with corner rank 2 or higher.  By the inductive hypothesis, 
\uprank{\graph{G}}{2}
has finite corner rank.  Now consider the corner ranking procedure applied to \graph{G}: we first remove the corner rank 1 vertices, leaving us with
\uprank{\graph{G}}{2}, and since 
\uprank{\graph{G}}{2} has finite corner rank, the corner ranking procedure terminates by assigning finite corner rank to \graph{G}.

\end{itemize}
\end{proof}

\noindent
From Theorem~\ref{thm_rank_capt_time} and Lemma~\ref{lemma_finiterank_dismantling} we conclude the following.

\begin{thm} \cite{NW83, Qui78}
A graph is cop-win if and only if it has a dismantling ordering.
\end{thm}

\section{Comparison of Cop Strategies} \label{sec_final_discussion}

This section is about relationships among different classes of optimal cop strategies. 
In the first part,  we show that  for cop-win graphs in general, 
the Catching strategies are a strict subset of the Lower Way strategies,
which are in turn, a strict subset of the set of all optimal strategies.
However,  in the second part, we describe a natural class of graphs on which, 
if the robber plays a Higher Way strategy, then
all optimal strategies are Lower Way strategies. We conclude with some questions raised by these results.

\subsection{Catching versus Lower Way versus Everything Else}

\begin{thm}\label{catnlw}  
There are graphs on which there is a Lower Way strategy
that is not a Catching strategy.
\end{thm}

\begin{proof}

We show that in the cop-win graph in Figure~\ref{lowmid} there is a Lower Way strategy that is not a Catching Strategy.  Suppose the cop starts in a standard position at vertex $c_1$. 
Wherever the robber places himself, the cop can $5$-corner him.
Suppose the robber starts at $r_1$, following a Higher Way strategy. In her first move, suppose the cop moves to $c_2$.  Since the robber is now 5-cornered, 
the cop satisfies the {\bf LW} condition of Definition~\ref{def_lw_catching}, so by Theorem~\ref{thm_LW_C_continuation}, can continue to play Lower Way.
However the cop has not $k$-caught the robber for any $k$, and so this is not a valid move in a Catching strategy.  
\end{proof}

\noindent

In a Catching strategy, the cop successively catches the $k$-projection of the robber for smaller and smaller values of $k$, and thus occupies vertices of smaller rank as the game goes on. It might be natural to assume that when the cop plays a Lower Way strategy, the cop starts at a high rank vertex and similarly moves to reduce her rank on each move. But this need not be the case, and in fact there are graphs for which the cop can start at low-rank vertices or can occupy lower-ranked vertices than the robber while still maintaining condition {\bf LW}. For example,  consider how the Lower Way strategy might continue in the proof of Theorem~\ref{catnlw} . The robber could move to $r_2$, followed by a cop move to $c_3$, followed by a robber move to $r_3$.  At this point, the cop could move to the rank 1 vertex $c_4$ in order to $3$-corner the robber. As another example, note that in Figure~\ref{onestart}, a cop playing a Lower Way (or Catching) strategy can start at the rank 1 vertex $v$.

\begin{figure}
\begin{center}
  \begin{tikzpicture}[every node/.style={circle,fill=black!20}]  
      \node (1) at (1,5) [label=left:$c_1$]{6};
      \node (2) at (3,5) {6};
     \node (3) at (-1.5,4) {5};
     \node (4) at (-1.5,3) {4};
     \node (5) at (-1.5,2) {3};
     \node (6) at (-1.5,1) {2};
     \node (7) at (-1.5,0) {1};
     \node (8) at (1,4) [label=left:$c_2$]{5};
     \node (9) at (1,3) [label=left:$c_3$]{4};
     \node (10) at (.8,2) {3};
     \node (11) at (.8,1) {2};
     \node (12) at (.8,0) {1};
     \node (13) at (2,0)[label=below:$c_4$]{1};
     \node (14) at (3,4) [label=right:$r_1$]{5};
     \node (15) at (3,3) [label=right:$r_2$]{4};
     \node (16) at (3.2,2) [label=right:$r_3$]{3};
     \node (17) at (3.2,1) [label=right:$r_4$]{2};
     \node (18) at (3.2,0) [label=right:$r_5$]{1};
     \node (19) at (5,2) {3};
     \node (20) at (5,1) {2};
     \node (21) at (5,0) {1};
     \node (22) at (6,4) {5};
     \node (23) at (6,3) {4};
     \node (24) at (6,2) {3};
     \node (25) at (6,1) {2};
     \node (26) at (6,0) {1};
    \foreach \from/\to in {1/2, 1/3, 3/4, 4/5, 5/6, 6/7, 1/8, 8/9, 9/10, 10/11, 11/12, 9/13, 8/2, 8/14, 9/14, 9/15, 2/14, 14/15, 15/16, 16/17, 17/18, 13/15, 13/16, 15/19, 19/20, 20/21, 2/22, 22/23, 23/24, 24/25, 25/26}
    \draw (\from) -- (\to);
    \end{tikzpicture}
\caption{
A graph with a Lower Way strategy which is not a Catching strategy.  In this Lower Way strategy, the cop can move to the rank 1 vertex $c_4$ while 3-cornering the robber.
}\label{lowmid}
\end{center}
\end{figure}
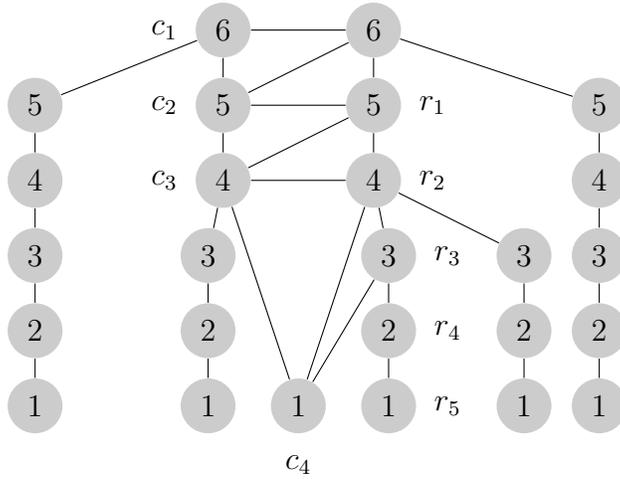

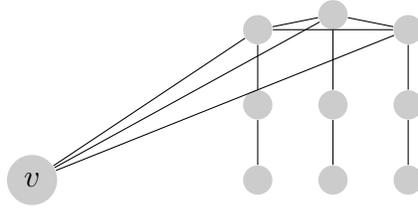
\begin{figure}
\begin{center}
  \begin{tikzpicture}[every node/.style={circle,fill=black!20}]  
      \node (n11) at (-2,-1) {$v$};
      \node(a21) at (1, -1) {};
      \node (n21) at (1,0)  {};
      \node(a31) at (2, -1) {};
   \node (n31) at (2,0)   {};
      \node(a41) at (3, -1) {};
      \node (n41) at (3,0)  {};
      \node (n12) at (1,1)   {};
      \node (n22) at (2,1.2) {};
      \node (n32) at (3,1) {};
    \foreach \from/\to in {n11/n12, n11/n22, n11/n32, n12/n22, n12/n32, n22/n32, n21/n12, n31/n22, n41/n32, n21/a21, n31/a31, n41/a41}
    \draw (\from) -- (\to);
    \end{tikzpicture}
\end{center}
\caption{A cop-win graph with a standard starting position of rank one (vertex $v$).}\label{onestart}
\end{figure}

We have shown that every Lower Way strategy is optimal.  However the converse is not true.

\begin{thm}
There are cop-win graphs with optimal strategies that are not Lower Way.
\end{thm}

\begin{proof}
Figure~\ref{nlw} shows a $0$-cop-win graph of rank 5, which by Theorem~\ref{thm_rank_capt_time} has capture time 5.  In addition to the optimal Lower Way strategies, there is an optimal strategy for this graph where the cop starts at the rank 1 vertex labeled $x$. No matter where the robber starts, he will not be $k$-cornered for any $k$, but we will show that the cop has an optimal strategy nonetheless. If the robber initially chooses $v_i$ where $i$ is even, the cop should move to $v_{14}$ and otherwise to $v_{13}$.  Note that an optimal robber is still not $4$-cornered, (e.g. a robber playing a Higher Way strategy would have started at $v_3$, $v_4$, $v_5$, or $v_6$) so condition {\bf LW} is not met. On the robber's next move, he must move to a $v_i$ for some value of $i$ greater than 4.  If the robber is at a vertex $v_i$ where $i$ is even, the cop should move to $v_4$, and otherwise the cop should move to $v_3$.  From there, the cop can force a win within 3 more moves, for a total of five, and thus we have described an optimal strategy that is not Lower Way.
\end{proof}

\begin{figure}
\begin{center}
  \begin{tikzpicture}[every node/.style={circle,fill=black!20}, scale=1.3]  
      \node (1) at (1,4) [label=above:$v_1$]{5};
      \node (2) at (2,4)  [label=above:$v_2$]{5};
   \node (3) at (0,3)   [label=right:$v_3$]{4};
   \node (4) at (-1,3)   [label=left:$v_5$]{4};
      \node (5) at (4,3) [label=right:$v_6$]{4};
      \node (6) at (3,3)   [label=left:$v_4$]{4};
      \node (7) at (-1,2) [label=right:$v_7$]{3};
      \node (8) at (4,2) [label=left:$v_8$]{3};
      \node (9) at (-1,1) [label=left:$v_{9}$]{2};
      \node (10) at (-3,1) [label=left:$v_{11}$]{2};
      \node (11) at (1,1) [label=left:$v_{13}$]{2};
      \node (12) at (2,1) [label=right:$v_{14}$]{2};
      \node (13) at (6,1) [label=right:$v_{12}$]{2};
      \node (14) at (4,1) [label=right:$v_{10}$]{2};
      \node (15) at (-1,0) [label=left:$v_{15}$]{1};
      \node (16) at (-3,0) [label=left:$v_{17}$]{1};
      \node (17) at (1.5,0) [label=left:$x$]{1};
      \node (18) at (6,0) [label=right:$v_{18}$]{1};
      \node (19) at (4,0) [label=right:$v_{16}$]{1};
    \foreach \from/\to in {1/2, 1/3, 1/4, 1/11, 1/12, 2/5, 2/6, 2/11, 2/12, 3/4, 3/7, 3/11, 4/7, 4/10, 5/6, 5/8, 5/13, 6/8, 6/12, 7/9, 8/14, 9/15, 10/16, 11/12, 11/17, 12/17, 13/18, 14/19}
    \draw (\from) -- (\to);
    \end{tikzpicture}
\caption{
A graph with an optimal cop strategy that is not Lower Way.
}\label{nlw}
\end{center}
\end{figure}
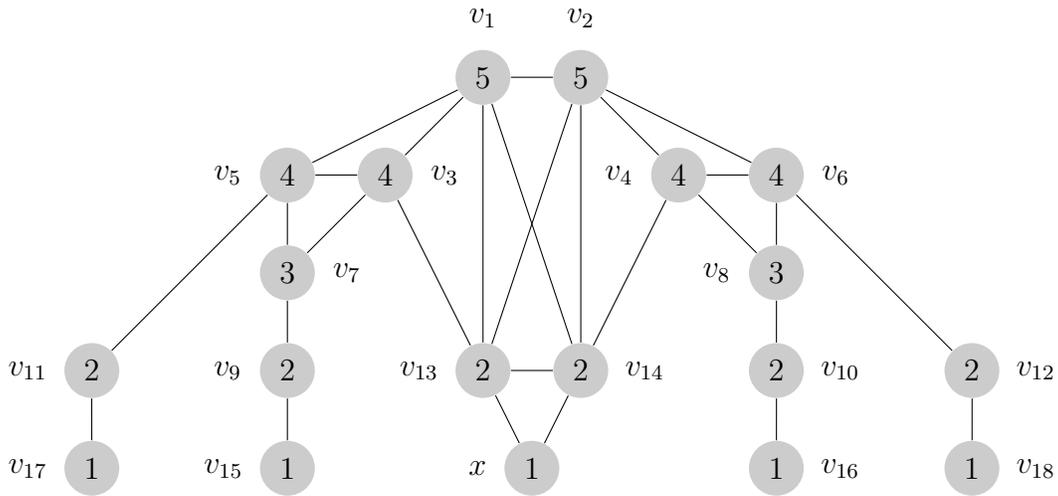

\subsection{Iteratively Twin-Free Graphs}

We describe a class of cop-win graphs on which all optimal cop strategies are
Lower Way strategies, when the robber plays a Higher Way strategy.  Further, on such graphs
no matter what strategy the robber plays, an optimal cop must strategy must begin with a standard initial placement.

\begin{defn}
Consider $k \ge 2$ and a graph \graph{G}. We say that two vertices $x$ and $y$ are \dword{k-twins} if they both have rank $k$ and are twins in 
\uprank{\graph{G}}{k - 1}.
\end{defn}

\noindent
For example, in the graph in Figure~\ref{nlw}, $v_4$ and $v_6$ are 4-twins, since they both have rank 4 and are twins in $G_3$. 

\begin{defn}
We say that a graph  is 
\dword{iteratively twin-free} if it has finite corner rank $\alpha$, and
there are no $k$-twins for $2 \le k \le \alpha$. 
\end{defn}

\noindent
In other words, a graph is \emph{iteratively twin-free} if each step of the corner ranking procedure (i.e. removing all the current strict corners), leaves a graph with no twins (except possibly at the current lowest rank).
For example, the graph in figure~\ref{nlw} is \emph{not} iteratively twin-free because after all the rank 1 and rank 2 vertices are removed, $v_4$ and $v_6$ are twins that are not at the lowest rank.
Such graphs fit naturally with the definition of corner rank.  They include the natural and more restrictive class of graphs:
\begin{quote}
Those finite rank graphs \graph{G} such that none of the graphs produced in the corner ranking procedure 
(i.e. 
$\uprank{\graph{G}}{1}, \ldots, \uprank{\graph{G}}{\alpha}$)
have twins, except possibly \uprank{\graph{G}}{\alpha}.
\end{quote}
Twins cause complications in some of the development, so avoiding them is interesting.  In fact, consider the similar ranking procedure in \cite{CNG14}:
In each step, to go from 
\uprank{\graph{G}}{k} to \uprank{\graph{G}}{k + 1}, involves first collapsing twins, and then removing all corners.  On this restrictive class of graphs, the procedure of \cite{CNG14} is simplified to just removing all corners (without concern for strict versus non-strict) at each step.

The following lemma mirrors Lemma~\ref{lemma_robber_response}. However in this version, the fact that the graph is iteratively twin-free allows the robber to play against the cop rather than the projection of the cop, moving to $k$-safe vertices rather than $k$-proj-safe vertices.

\begin{lemma} \label{lemma_twinfree_robber_response}
Suppose \graph{G} is an iteratively twin-free graph, the cop is at some vertex $c_0$, and the robber is at some $k$-safe vertex where $k \ge 2$. Suppose the cop moves from $c_0$ to $c_1$.

\begin{enumerate}

\item
The robber has a move to a $(k-1)$-safe vertex.

\item 
Furthermore,
if the robber is \emph{not} cornered by $c_1$ in 
$\uprank{\graph{G}}{k} \cup c_1$, then the robber has a move to a $k$-safe vertex.

\end{enumerate}

\end{lemma}

\begin{proof}  

Since the robber was $k$-safe, the cop move to $c_1$ did not catch him.
The second part follows immediately from the definition of not being cornered.
Now consider the first part.
  Suppose the robber is at a vertex $r$ of rank $k^* \ge k$.
Assume for contradiction there is no $(k-1)$-safe move for the robber. Then $c_1$ corners $r$ in 
$\uprank{\graph{G}}{k-1} \cup c_1$. 
Let $c_1' \in F_{k-1}(c_1)$. Lemma~\ref{projnbr} implies that $\clnbk{k-1}[c_1] \subseteq \clnbk{k-1}[c_1']$, and thus $c_1'$ also corners $r$ in 
\uprank{\graph{G}}{k-1}.
Since $r$ has rank $k^* \ge k$, this cornering cannot be strict, so $c_1'$ and $r$ must  be twins in 
\uprank{\graph{G}}{k-1}.
However since  $r$ has rank $k^*$, $c_1'$ also has rank $k^*$, so $r$ and $c_1'$ are both of rank $k^* \ge k$ and twins in 
\uprank{\graph{G}}{k^*-1}, i.e. $r$ and $c_1'$ are $k^*$-twins,  contradicting the assumption that \graph{G} is iteratively twin-free.
\end{proof}

We already know that the standard start positions can always be followed by optimal cop play.  For iteratively twin-free graphs we get the converse: to be optimal the cop must start standard.
\begin{thm} \label{thm_optimal_start}
In an iteratively twin-free cop-win graph the initial placement by the cop is optimal if and only if it is standard.
\end{thm}
\begin{proof}
The backwards direction follows from the existence of Lower Way strategies, which are optimal and whose initial placement can be anything that is standard. Consider the forward direction.  Suppose the cop does not start standard.  In the $0$-cop-win case, the robber can start $\alpha$-safe, and in the $1$-cop-win case, the robber can start $(\alpha-1)$-safe.  Then by Lemma~\ref{lemma_twinfree_robber_response}, the robber can play to reduce his rank at most 1 per turn; in either case, this prolongs the game at least one cop move more than the optimal capture time.

\end{proof}

\noindent
On iteratively twin-free graphs, we classify all optimal cop strategies given the assumption that the robber plays a Higher Way strategy.

\begin{thm}
Suppose the robber plays a Higher Way strategy
on an iteratively twin-free graph.  Then a cop strategy is optimal if and only if it is Lower Way.

\end{thm}
\begin{proof}

The backward direction was noted in Theorem~\ref{lwhwopt}.  Consider the forward direction, just for $0$-cop-win graphs
(the proof for $1$-cop-win graphs is virtually identical).
By Theorem~\ref{thm_optimal_start}, the cop must start standard, consistent with a Lower Way strategy.
Suppose the graph has finite corner rank $\alpha$.
By Lemma~\ref{lemma_robber_start} (recalling that being $k$-proj-safe safe implies being $k$-safe), the robber will start at an $(\alpha-1)$-safe vertex. By Lemma~\ref{lemma_twinfree_robber_response} (Part 1) the robber will move down at most one rank per a turn.  Thus right before the cop makes her $t^{th}$ move ($t \ge 1$), the robber will be at some rank $\ge \alpha - t$.  At this point, if the cop does not follow a Lower Way strategy and corner the robber in 
\uprank{\graph{G}}{\alpha - t}, then by Lemma~\ref{lemma_twinfree_robber_response} (Part 2) the robber can safely remain at rank $\ge \alpha - t$.  But now the cop has made $t$ moves and by repeated application of Lemma~\ref{lemma_twinfree_robber_response} (Part 1), after $(\alpha - t - 1)$ more cop moves, the robber will be $1$-safe, i.e. not adjacent to the cop, thus requiring at least 2 more cop moves for the cop to catch the robber.  Thus by \emph{not} following a Lower Way strategy, the cop was forced to make at least $t + (\alpha - t - 1) + 2 = \alpha + 1$ moves, and the strategy is not optimal.
\end{proof}

\subsection{Conclusions}

We conclude this paper with a somewhat vague research proposal:
\begin{quote}
\emph{Find a ``nice'' characterization of the optimal strategies for the cop and
for the robber.}  
\end{quote}

\noindent
The Catching strategies are a nice description of a number of optimal cop strategies, 
but not all of them.  The Lower Way strategies are a nice description of 
a wider class of cop strategies, but still not all of them.
So we have nice descriptions which are not comprehensive.

On the other hand, we can give comprehensive descriptions which are not nice.  Hahn and MacGillivary \cite{HG2006} define a rank function (right before their Lemma 4) on the collection of
all possible game positions (i.e. the pairs consisting of the cop location and the robber location); this rank builds
on the ``$\le_{\gamma}$-relation'' from \cite{NW83}. 
They show that the rank of a position is equal to
the number of cop moves required to win from that position.
Thus, we could use their rank function to describe a class of cop strategies that is comprehensive:
\emph{The cop starts at a location that minimizes the rank over all possible robber responses; then on subsequent moves the cop chooses to reduce the rank.}  An analogous approach would work for the robber strategies.
However we would not 
call such a description ``nice'' because it does
not give much insight into the specific game of Cops and Robber.
The basic point is that most any game we care about
can be described by a similar rank function and a similar class of rank-reducing optimal strategies.

Thus to be considered a nice strategy characterization, we want the description to be  tied to the specific structure of the game of Cops and Robber.  The notion of corner rank is tied to the specific structure of the game, which is why we find the Catching and Lower Way strategies to be nice.  Thus, we can restate our proposal:
\begin{quote}
\emph{Find a ``nice'' characterization, in terms of corner rank, of the optimal strategies for the cop and
for the robber.}  
\end{quote}

\noindent
Our hope would be to extend the definitions of the Lower Way strategies and Higher Way strategies in some nice way, 
so that they encompassed all the optimal strategies.

\section{Acknowledgments}
The first author was supported by the Westminster College McCandless Research Award.
The second author was supported by a PSC-CUNY Research Award (Traditional A).


\begin{thebibliography}{999}











\bibitem{BGHK2009} A. Bonato, P. Golovach, G. Hahn, and
J. Kratochvil, The capture time of a graph,
 Discrete Math. 309 (2009), 5588-5595.


\bibitem{BN11} A. Bonato and R. Nowakowski, The game of cops and robbers on graphs, AMS Student Mathematical Library, 2011.

\bibitem{CNG14} N. Clarke, S. Finbow, and G. MacGillivary, A simple method for computing the catch time, Ars Mathematica Contemporanea, 7 (2014) 353--359.

\bibitem{CN2001} N. Clarke and R. Nowakowski, Cops, robber, and trap, Utilitas Mathematica 60: 91-98, 2001.







\bibitem{HG2006} G. Hahn and G. MacGillivary, 
A note on $k$-cop, $l$-robber games on graphs, Discrete Math. 306 (2006),
 2492-2497.













\bibitem{NW83} R. Nowakowski and P. Winkler, Vertex-to-vertex pursuit  in a graph, Discrete Math. 43 (1983), 235-239.


\bibitem{OO16} D. Offner and K. Ojakian, Corner ranking, realizable vectors, and extremal cop-win graphs, submitted.


\bibitem{Qui78} A. Quilliot, Jeux et pointes fixes sur les graphes,  Ph.D. Dissertation, Universit\'e de Paris VI, 1978.





\end{thebibliography}
\end{document}